\documentclass[psamsfonts,reqno,12pt]{amsart}

\usepackage{amssymb,amsfonts}
\usepackage[all,arc]{xy}
\usepackage{enumerate}
\usepackage{mathrsfs}
\usepackage{hyperref}
\usepackage{mathtools}
\usepackage{nicefrac}
\usepackage[margin=1.0in]{geometry} 
\usepackage{mathtools}
\usepackage{tikz} 
\usetikzlibrary{angles,quotes,calc,intersections}

\newtheorem{thm}{Theorem}[section]
\newtheorem{cor}[thm]{Corollary}
\newtheorem{prop}[thm]{Proposition}
\newtheorem{lem}[thm]{Lemma}

\theoremstyle{definition}

\theoremstyle{remark}
\newtheorem{rem}[thm]{Remark}

\numberwithin{equation}{section}

\newcommand{\N}{\mathbb{N}}

\newcommand{\R}{\mathbb{R}}

\newcommand{\eps}{\varepsilon}

\newcommand{\inj}{\xhookrightarrow{{\kern 1em}}}

\title{Quantitative marked length spectrum rigidity for surfaces}

\author{Karen Butt}

\begin{document}

\begin{abstract} 
We consider a closed negatively curved surface $(M, g)$ with marked length spectrum sufficiently close (multiplicatively) to that of a hyperbolic metric $g_0$ on $M$. 
We show there is a smooth diffeomorphism $F:M \to M$ with derivative bounds close to 1, depending on the ratio of the two marked length spectrum functions.
This is a two-dimensional version of our main result in \cite{butt22}. 
\end{abstract}

\maketitle


\section{Introduction} 

The \emph{marked length spectrum} $\mathcal{L}_g$ 
of a closed Riemannian manifold $(M, g)$ of negative curvature is a function on the free homotopy classes of closed curves in $M$ which assigns to each class the length of its unique $g$-geodesic representative. 

In this paper, we consider the extent to which $\mathcal{L}_g$ determines $g$ when $M$ is a surface. 
It was proved by Otal \cite{otalMLS} and Croke \cite{croke90} that for negatively curved $g$ (and more generally by Guillarmou--Lefeuvre--Paternain \cite{GLP} for $g$ Anosov), that the length function $\mathcal{L}_g$ determines $g$ up to isometry. This is known as \emph{marked length spectrum rigidity}. (In higher dimensions this is a major open question \cite[Conjecture 3.1]{burnskatok}, but there are partial results due to Hamenst{\"a}dt \cite{ham99} and Guillarmou--Lefeuvre \cite{GL19}.)

When rigidity is know to hold, it is natural to ask if one can ``approximately recover" the metric from ``approximate length data". 
For instance, let $\mathcal{C}$ denote the set of free homotopy classes of closed curves in $M$, fix $\eps > 0$, and suppose that for all $c \in \mathcal{C}$ we have
\begin{equation}\label{MLSass}
1 - \eps \leq 
\frac{\mathcal{L}_g(c)}{\mathcal{L}_{g_0}(c)} 
\leq 1 + \eps.
\end{equation}
Is there a way to quantify how ``close" the metrics $g$ and $g_0$ are in terms of $\eps$? 

Such stability questions were considered by Guillarmou--Lefeuvre \cite{GL19} and Guillarmou--Knieper--Lefeuvre \cite{GKL22} in the setting of two metrics which are sufficiently close in some $C^k$ topology, and by the author in \cite{butt22} in the settings of the rigidity results of both Otal \cite{otalMLS} and Hamenst{\"a}dt \cite{ham99}. 

In the case of surfaces \cite[Theorem A]{butt22}, we used ideas from \cite{otalMLS} together with geometric convergence arguments to show that for pairs of negatively curved metrics $g$ and $g_0$, with controlled injectivity radius and sectional curvature bounds, and with marked length spectra $\eps$-close as in \eqref{MLSass}, there exists an $L(\eps)$-biLipschitz map from $(M, g) \to (M, g_0)$ with $L(\eps) \to 1$ as $\eps \to 0$ \cite[Theorem A]{butt22}. Moreover, $L$ depends only on concrete Riemannian data of $g$ and $g_0$, namely, the injectivity radii, diameters, and bounds on the sectional curvatures. 
In the special case of surfaces of constant negative curvature, a related question was also considered by Thurston \cite{thurston98}. He showed that the best possible Lipschitz constant for a map $F: (M, g) \to (M, g_0)$ in the same homotopy class as $f$ is precisely $\sup_{c \in \mathcal{C}} \frac{\mathcal{L}_g(c)}{\mathcal{L}_{g_0}(c)}$; however, the Lipschitz constant of $F^{-1}$ is not well-controlled.

In the higher dimensional setting, we obtained a much stronger result under the additional assumption that $g_0$ is locally symmetric \cite[Theorem B]{butt22}, refining the rigidity result in \cite{ham99}.
More precisely, we were able to explicitly estimate how $L(\eps)$ depends on $\eps$, and the estimates depend additionally only on concrete Riemannian data about $g$ and $g_0$ such as the diameter, injectivity radius and sectional curvature bounds.

In this paper, we prove a result which complements our above two results. Namely, we work in dimension 2 under the additional assumption that $g_0$ is locally symmetric (equivalently, hyperbolic, meaning of constant negative curvature). 
If $\mathcal{L}_g$ and $\mathcal{L}_{g_0}$ are close as in \eqref{MLSass}, 
we obtain explicit estimates for the $C^0$ distance between $g$ and $g_0$ in terms of $\eps$, and our estimates are uniform for metrics with a fixed lower bound on the injectivity radius and fixed upper and lower sectional curvature bounds:

\begin{thm}\label{mainthm}
Let $M$ be a smooth closed orientable surface of genus $G \geq 2$.
Let $g$ be a $C^2$ Riemannian metric on $M$ with injectivity radius bounded below by $i_0$ and Gaussian curvature contained in the interval $[-b^2, -a^2]$. 
Let $g_0$ be another metric on $M$ of constant curvature $\kappa < 0$. There exists $\eps_0 = \eps_0(\kappa, i_0, a, b)$ sufficiently small such that the following holds: let $\eps \leq \eps_0$ and suppose the marked length spectra of $g$ and $g_0$ are $\eps$-close  as in \eqref{MLSass}.
Then there exists a smooth diffeomorphism $F: (M, g) \to (M, g_0)$, homotopic to the identity, together with a constant $C = C(G, \kappa, i_0, a, b)$ such that
\[
1 - C \eps^{1/8} \leq \frac{\Vert DF(v) \Vert_g}{\Vert v \Vert_{g_0}} \leq  1 + C \eps^{1/8}
\]
for all $v \in TM$. 
\end{thm}

\begin{rem}
We can obtain the above conclusion by only assuming a version of \eqref{MLSass} holds on a sufficiently large finite subset of $\mathcal{C}$ using our earlier work \cite[Theorem 1.2]{butt22finite} or the improved estimate subsequently obtained by Cantrell--Reyes \cite{CR24}. 
More precisely, \cite[Theorem 1.1]{CR24} gives that there exist constants $C$ and $L_0$ depending only on $G, i_0, a, b, \kappa$ such that the following holds. Suppose
\[
1 - \eps' \leq 
\frac{\mathcal{L}_g(c)}{\mathcal{L}_{g_0}(c)} 
\leq 1 + \eps'
\]
for all $c \in \mathcal{C}$ with $\mathcal{L}_{g_0}(c) \leq L$. 
Then \eqref{MLSass}, and hence Theorem \ref{mainthm}, holds with $\eps = \eps' + C/L$. 
\end{rem}

\subsection*{Strategy of the proof} The overall approach is similar to the proof of \cite[Theorem B]{butt22}. There, the $\eps = 0$ case is due to Hamenst{\"a}dt \cite{ham99}, using work of Besson--Courtois--Gallot. More precisely, Hamenst{\"a}dt's result shows in particular that for $g_0$ locally symmetric, $\mathcal{L}_g = \mathcal{L}_{g_0}$ implies ${\rm Vol}(g) = {\rm Vol}(g_0)$. 
Then, since $\mathcal{L}_g$ also determines the topological entropy $h(g)$ of the geodesic flow, marked length spectrum rigidity follows from the entropy rigidity theorem of Besson--Courtois--Gallot \cite{BCGGAFA}. 
This theorem says that for a locally symmetric space $(M, g_0)$ of dimension at least 3, any other metric $g$ on $M$ with 
$h(g) = h(g_0)$ and ${\rm Vol}(g) = {\rm Vol}(g_0)$ must be isometric to $g_0$. 

In dimension 2, this entropy rigidity theorem is false, though work of Katok \cite{katok} shows that if $g_0$ is a hyperbolic metric, the equalities $h(g) = h(g_0)$ and ${\rm Vol}(g) = {\rm Vol}(g_0)$ imply $g$ is also hyperbolic,
though not necessarily isometric to $g_0$.
Indeed, Katok proved that entropy rigidity holds in the conformal class of a locally symmetric metric, and in dimension 2 every metric $g$ on a closed surface of genus at least 2 is conformally equivalent to a hyperbolic metric.  

In the case $\eps = 0$, our proof of Theorem \ref{mainthm} reduces to the following argument: by the previous two paragraphs, $\mathcal{L}_g = \mathcal{L}_{g_0}$ implies that $g$ is hyperbolic, and it is a classical result (using Fenchel--Nielsen coordinates for Teichm{\"u}ller space) that two hyperbolic metrics with the same marked length spectrum are isometric. 
In Section \ref{sec:setup}, we give a more detailed outline of the argument.
In Section \ref{sec:conf-fact}, we show an effective version of Katok's theorem, and in Section \ref{sec:9G-9}, we prove the main theorem for two hyperbolic metrics using Fenchel--Nielsen coordinates.

\subsection*{Acknowledgements} 
I thank James Marhsall Reber for useful comments on an earlier draft of this paper.
I also thank Aaron Calderon and Nicholas Wawrykow for helpful conversations.
This research was supported in part by NSF grant DMS-2402173.  
\section{Initial setup and outline of the proof}\label{sec:setup}

We first note that two metrics with almost the same marked length spectrum as in \eqref{MLSass} have almost the same area and almost the same topological entropy. 
Indeed, let $h(g)$ denote the topological entropy of the geodesic flow of $g$. Then by 
\cite{margulis1969} we have:
\begin{equation*}\label{mar}
h(g) = \lim_{t \to \infty} \frac{1}{t} \log P_g (t),
\end{equation*}
where $P_g(t) = \# \{ \gamma \, | \, l_g(\gamma) \leq t \}.$  
Using \eqref{MLSass}, it follows that 
\begin{equation}\label{eq:ent}
(1 + \eps)^{-1} h(g_0) \leq h(g) \leq (1 - \eps)^{-1} h(g_0).
\end{equation}
Next, let $A(g)$ denote the total area of $(M, g)$. Note that $\eqref{MLSass}$ can be reformulated as $\mathcal{L}_{(1 - \eps)^2g_0} \leq \mathcal{L}_g \leq \mathcal{L}_{(1 + \eps)^2 g}$. 
Hence \cite[Theorem 1.1]{CD2004} gives
\begin{equation}\label{eq:area}
(1 - \eps)^2 A(g_0) \leq A(g) \leq (1 + \eps)^2 A(g_0).
\end{equation}

Now note that for any number $\kappa < 0$ and any smooth Riemannian metric $g$ on a smooth closed orientable surface $M$ of genus $G \geq 2$, 
the uniformization theorem for Riemann surfaces implies that there exists a smooth function $u: M \to \R$ so that $g_1:=e^{-2u} g$ has constant negative curvature $\kappa$. 
Since $g_0$ has constant curvature $\kappa$, we have
 $h(g_0) =h(g_1) = \sqrt{- \kappa}$ and $A(g_0) =A(g_1) = 2 \pi (2 - 2G)/{ \kappa}$.
So the inequalities \eqref{eq:ent} and \eqref{eq:area} above hold with $g_0$ replaced by $g_1 = e^{-2u} g$. In other words, $g$ has almost the same entropy and almost the same area as a hyperbolic metric in its conformal class. 

If $g_1$ is hyperbolic and $g = e^{2u} g_1$ is a conformally equivalent metric with the same area and topological entropy as $g_1$, then $u \equiv 0$ \cite[Corollary 2.5]{katok}.
If the areas and entropies of $g$ and $g_1$ do not coincide, then
in our notation, \cite[Corollary 2.4]{katok} reads:
\begin{equation}\label{eq:katok}
h(g) \geq \left( \int_M e^u \frac{dA_{g_1}}{A(g_1)} \right)^{-1} h(g_1) \geq \left( \frac{A(g)}{A(g_1)} \right)^{1/2} h(g_1).
\end{equation}
Note that by the Cauhcy--Schwarz inequality, we have
\[
\int_M e^u \frac{dA_{g_1}}{A(g_1)} \leq \sqrt{\frac{A(g)}{A(g_1)}},
\]
and that \eqref{eq:katok} rearranges to
\begin{equation}\label{ineq:almost-cs}
1 \geq \frac{\int_M e^u \frac{dA_{g_1}}{A(g_1)}}{\sqrt{A(g)/A(g_1)}} \geq \frac{h(g_1)}{h(g)} \sqrt{\frac{A(g)}{A(g_1)}} \geq \frac{1- \eps}{1 + \eps}.
\end{equation}

In other words, equality almost holds in the Cauchy--Schwarz inequality. 
In Section \ref{sec:conf-fact}, we show this implies the $C^0$ norm of $u$ is small in terms of $\eps$.
This implies that we only need to check the main theorem with $g$ replaced by $g_1$, that is, both metrics are hyperbolic. 
We do this in Section \ref{sec:9G-9}, using Fenchel--Nielsen coordinates and maps constructed in \cite[Chapter 3]{buser}.

\section{Controlling the conformal factor}\label{sec:conf-fact}
We begin by establishing some a priori $C^0$ and $C^1$ bounds on $u$ in terms of concrete geometric data of $g$.

\begin{prop}\label{prop:apriori}
Let $(M, g)$ as in the hypotheses of Theorem \ref{mainthm}. 
Let $u:M \to \R$ be such that $g_1 = e^{-2u} g$ is a metric of constant curvature $\kappa < 0$.
Then there exist constants $C_0 = C_0(a, b, \kappa)$ and $C_1 =  C_1(i_0, a, b, \kappa)$ so that:
\begin{enumerate}
 \item $\Vert u \Vert_{C^0(M)} \leq C_0$ 
 \item $\Vert \nabla u \Vert_{C^0(M)} \leq C_1$. 
 \end{enumerate}
\end{prop}

\begin{proof}
Note that the Gaussian curvatures of conformally related metrics are related as follows:
\[
K_{g} = K_{e^{2u}g_1} = e^{-2u} (- \Delta_{g_1} u  + K_{g_0}). 
\]
(See for instance \cite[Lemma 5.3]{CK04}.)
This rearranges to
\begin{equation}\label{eq:poisson}
\Delta_{g_1} u = \kappa - e^{2u} K_{g} .
\end{equation}
To obtain an upper bound on $\Vert u \Vert_{C^0(M)}$, let $p \in M$ such that $\max_{M} u = u(p)$. Then $\Delta_{g_1} u(p) \leq 0$, which means
$e^{2u(p)} \leq \frac{\kappa}{K_{g}(p)}$. This proves the desired $C^0$ bound.

In order to prove the $C^1$ bound, we first claim the injectivity radius $i_1$ of $g_1$ can be bounded from below by a constant depending only on $\kappa, a, b, i_0$. 
Let $\gamma(t)$ be a $g_1$--unit-speed parametrization a closed $g_1$-geodesic such that $i_1 = \frac{1}{2}l_{g_1}(\gamma)$. We then have 
\begin{equation}\label{eq:injrad}
l_{g_1}(\gamma) = \int_{0}^{l_{g_1}(\gamma)} e^{-u(\gamma(t))} g(\gamma'(t), \gamma'(t))^{1/2} \, dt \geq e^{-C_0(a,b, \kappa)} l_g (\gamma) \geq e^{-C_0(a,b, \kappa)} 2 i_0,
\end{equation}
which proves the claim.

To bound $\Vert \nabla u \Vert_{C^0(M)}$, we will examine the equation \eqref{eq:poisson} in local coordinates. 
Let $f = \kappa - e^{2u} K_g$. We have already established that $\Vert f \Vert_{C^0(M)} \leq C$ for some $C$ depending only on $\kappa, a, b$. 
Fix $p \in (M, g_1)$ and consider a ball of radius $i_1$ centered at $p$.
This ball is isometric to a ball of radius $i_1$ centered at the origin in the Poincar{\'e} disk model. 
In these coordinates, we have that $g_1$ is of the form $e^{\lambda(x, y)} (dx^2 + dy^2)$. (Concretely, $e^{\lambda(x, y)}  =  c(\kappa) {(1 - (x^2 + y^2))^{-2}}$, where $c(\kappa)$ is an explicit constant depending on $\kappa$; hence, $c(\kappa) \leq e^{\lambda} \leq C = C(\kappa, a, b, i(g_1))$.) 
Thus, we can write $\Delta_{g_1}  = e^{-\lambda(x, y)} \Delta_{\R^2}$. 

This means that on the hyperbolic ball of radius $i(g_1)$, we have  \eqref{eq:poisson} is of the form
$\Delta_{\R^2} u = e^{- \lambda} f$.
In other words, $u$ solves the above Poisson equation on a Euclidean ball of radius $r$, where $r$ depends explicitly on $i_1$. 
In this setting, there are standard estimates for the gradient of $u$. For instance,
\cite[(3.15)]{GT} gives
\[
\Vert \nabla_{\R^2} u(p) \Vert_{\R^2} \leq \frac{2}{c} \Vert u \Vert_{C^0} + \frac{c}{2} \Vert e^{-2 \lambda} f \Vert_{C^0},
\]
where $c$ depends only on $i_1$. 
Noting that $\Vert \nabla_{\R^2} u \Vert_{\R^2} = e^{\lambda/2} \Vert \nabla_{g_1} u \Vert_{g_1}$
completes the proof.
\end{proof}

In the following two lemmas, we make precise the idea that if equality almost holds in the Cauchy--Schwarz inequality for some function $\rho$, then $\rho$ is almost constant.

\begin{lem}\label{lem:cs1}
Let $(M, g)$ be a closed Riemannian manifold and let $\rho \in L^2(M, {\rm vol}_g)$. Let $\bar \rho$ denote $\int_M \rho \, d {\rm vol}_g$. 
Suppose there is some $0 < c < 1$ so that
\begin{equation}\label{almost-cs-2}
c \left( \int_M \rho^2 \, d{\rm vol}_g \right)^{1/2} \leq \int_M \rho \, d{\rm vol}_g.
\end{equation}
Then
\[
\int_{M} |\rho - \bar \rho| \, d {\rm vol}_g \leq \sqrt{1 - c^2} \Vert \rho \Vert_{L^2(M)}. 
\]
\end{lem}

\begin{proof}
Let $\langle \cdot , \cdot \rangle$ denote the scalar product on $L^2(M, {\rm vol}_g)$. 
Consider the non-negative quadratic polynomial 
$q(t) = \langle \rho + t, \rho + t \rangle = t^2 + 2t \langle \rho, 1 \rangle + \langle \rho, \rho \rangle$. 
Let $D$ be its discriminant. 
Then the minimum value of $q(t)$ is $-D/4 = \langle \rho, \rho \rangle - \langle \rho , 1 \rangle^2$, and the minimum occurs at $t_0 = - \langle \rho, 1 \rangle = -\bar \rho$.
By hypothesis, $-D/4 \leq \delta := (1-c^2) \langle \rho, \rho \rangle$. 
This means $\sqrt{q(t_0)} = \Vert \rho - \bar \rho \Vert_{L^2} \leq \sqrt{\delta}$.
Finally, applying the Cauchy--Schwarz inequality gives $\int_M |\rho - \bar \rho| \, d {\rm vol}_g  \leq \Vert \rho - \bar \rho \Vert_{L^2} \leq \sqrt{\delta}$, which completes the proof. 
\end{proof}

Assuming additional control on $\rho$ and $M$, we can upgrade the previous result as follows. 

\begin{lem}\label{lem:cs2}
Let $(M^n, g)$ be a closed Riemannian maniold with injectivity radius at least $i_0$ and let $C(n)$ denote the volume of the unit ball in $\R^n$.
Let $\eta(n, i_0) = C(n) i_0^n$.
Let $\rho: M \to \R$ be a $L$-Lipschitz function.
Then for any $c$ sufficiently close to 1 so that $\delta : = (1 - c^2) \Vert \rho \Vert^2_{L^2(M)} \leq \eta^4$, the following holds:
Suppose $\rho$ satisfies \eqref{almost-cs-2} for $c$ as above. 
Then there exists a constant $C' = C'(n, L)$ so that
\[
\Vert \rho - \bar \rho \Vert_{C^0(M)} \leq C' \delta^{1/4n}.
\]
\end{lem}

\begin{proof}
The idea of the proof is similar to that of \cite[Proposition 3.18]{butt22}. 
First note that by the previous lemma, we have $\phi: = |\rho - \bar \rho|$ satisfies $\int_M \phi \, d{\rm vol}_g \leq \sqrt{\delta}$. 
Fix $\omega = \delta^{1/4} > 0$. 
Now let $B \subset M$ be the set where $\phi:= |\rho - \bar \rho| > \omega$.
An easy argument (see, for instance, \cite[Lemma 3.19]{butt22}) shows the measure of $B$ is at most $\delta^{1/2}/ \omega = \delta^{1/4}$.  

Since $\delta^{1/4} \leq C(n) i_0^n$, \cite[Lemma 3.20]{butt22} implies that for every $q \in B$ there exists $p \in M \setminus B$ with $d(p, q) \leq C(n) \delta^{1/4n}$. 
Then 
\[
\phi(q) \leq \phi(p) + |\phi(p) - \phi(q)| \leq \delta^{1/4} + L C(n) \delta^{1/4n} \leq C' \delta^{1/4n}
\]
for some $C' = C'(n, L)$. 
\end{proof}

We can now apply this to show:
\begin{prop}\label{prop:conf-fact}
Let $M$, $g$ and $g_1 = e^{-2u} g$ as in Proposition \ref{prop:apriori}. 
Then there exists $\eps_0 = \eps_0(\kappa, i_0, a, b)$ sufficiently small such that the following holds: 
Suppose that $\rho :=  e^{u}$ satisfies \eqref{ineq:almost-cs} for $\eps \leq \eps_0$.
Then there exists $C = C(i_0, a, b, \kappa)$ such that 
\[ 1 - C \eps^{1/8} \leq \Vert \rho \Vert_{C^0(M)} \leq 1 + C \eps^{1/8}.\]
\end{prop}

\begin{proof}
First, we note that 
\begin{equation}\label{eq:rhoL2}
\Vert \rho \Vert_{L^2} = \sqrt{A(g)/A(g_1)} \in [1 - \eps, 1 + \eps]
\end{equation}
and that 
\begin{equation}\label{eq:24}
c \Vert \rho \Vert_{L^2} \leq \bar \rho \leq \Vert \rho \Vert_{L^2}
\end{equation}
for $c = (1 - \eps)/(1+ \eps)$. See \eqref{ineq:almost-cs} above.

Let $\eta$ as in the statement of Lemma \ref{lem:cs2}. 
Let $\eps_0$ sufficiently small such that $\eps \leq \eps_0$ implies 
\[
1 - \frac{(1 - \eps)^2}{(1+ \eps)^2} \leq \frac{\eta^4}{ \Vert \rho \Vert^2_{L^2}}. 
\]
We claim that $\eps_0$ depends only on $\kappa, i_0, a, b$. Indeed, $\Vert \rho \Vert_{L^2}$ is controlled by \eqref{eq:rhoL2}. 
Moreover, $\eta$ depends only on the dimension $n = 2$ and the injectivity radius $i_1$ of $g_1$, which depends only on $i_0, \kappa, a, b$.

For $\eps \leq \eps_0$, Lemma \ref{lem:cs2} gives $\Vert \rho - \bar \rho \Vert \leq C \delta(\eps)^{1/8}$ for $\delta(\eps) = 1  - \frac{(1 - \eps)^2}{(1+ \eps)^2} = \frac{4 \eps}{(1 + \eps)^2} \leq 4 \eps$ and $C$ depending on a Lipschitz bound for $u$. By Proposition \ref{prop:apriori} part (2), this Lipschitz bound can be controlled by a constant $C_1 = C_1(i_0, a, b, \kappa)$. This shows $\Vert \rho - \bar \rho \Vert_{C^0(M)} \leq C \eps^{1/8}$ for some $C = C(i_0, a, b, \kappa)$. 
Finally, by  \eqref{eq:rhoL2} and \eqref{eq:24}, we have $|\bar \rho - 1| \leq 3 \eps$, which completes the proof. 
\end{proof}

To conclude this section, we record two important corollaries of the previous proposition.
The first shows that for $g$ and $g_1$ conformally equivalent and $g_1$ hyperbolic, almost equality of areas and topological entropies implies the identity map $M \to M$ is ``almost an isometry". 

\begin{cor}\label{cor:id-Lip}
Let $(M, g)$ as in Theorem \ref{mainthm} and let $g_1 = e^{-2u} g$ be a metric of constant curvature $\kappa < 0$.
Suppose $e^u$ satisfies \eqref{eq:katok}.  
Let $F: (M, g_1) \to (M, g)$ denote the identity map. 
Then there is a constant $C = C(i_0, a, b, \kappa)$ such that 
\[
1 - C \eps^{1/8} \leq \frac{\Vert DF(v) \Vert_g}{\Vert v \Vert_{g_1}} \leq 1 + C \eps^{1/8}
\] 
for all $v \in TM$. 
\end{cor}

The second is that the marked length spectrum of $g_1$ is close to the marked length spectrum of the original hyperbolic metric $g_0$. 
\begin{cor}\label{cor:MLS}
Let $M$, $g$, $g_1$, and $C$ as in the previous corollary. 
Let $g_0$ as in Theorem \ref{mainthm}.
Then
\[
(1 - \eps) (1 - C\eps^{1/8}) \leq \frac{\mathcal{L}_{g_0}(c)}{\mathcal{L}_{g_1}(c)} \leq (1 + \eps) (1 + C \eps^{1/8})
\]
for all $c \in \mathcal{C}$. 
\end{cor}

\begin{proof}
By \eqref{MLSass}, it suffices to show that 
\[
 1 - C\eps^{1/8} \leq \frac{\mathcal{L}_g(c)}{\mathcal{L}_{g_1}(c)} \leq 1 + C \eps^{1/8}
\]
for all $c \in \mathcal{C}$. 
The argument is very similar to \eqref{eq:injrad}. 
\end{proof}

\section{Effective $9G-9$ Theorem}\label{sec:9G-9}

Corollary \ref{cor:id-Lip} shows that in order to prove Theorem \ref{mainthm}, it suffices to find an ``almost isometry" between $g_1$ and $g_0$.
 Corollary \ref{cor:MLS} shows that $g_1$ and $g_0$ have multiplicatively close marked length spectra.
 As such, we have reduced Theorem \ref{mainthm} to the case where both metrics are hyperbolic, which we now prove. 

\begin{thm}\label{thm:hyp}
Let $M$ be a smooth closed orientable surface of genus $G \geq 2$. Let $\mathcal{C}$ denote the set of free homotopy classes of close curves in $M$. 
Let $g_0$ and $g_1$ be two hyperbolic metrics on $M$ with injectivity radii bounded below by $i_0$. Suppose that there is some $\eps > 0$ so that their marked length spectra satisfy
\begin{equation}\label{eq:MLShyp}
1 - \eps \leq 
\frac{\mathcal{L}_{g_1}(c)}
{\mathcal{L}_{g_0}(c)} 
\leq 1 + \eps
\end{equation}
for all $c \in \mathcal{C}$.
Then there exists a smooth map $F: (M, g_0) \to (M, g_1)$, homotopic to the identity, together with a constant $C$, depending only on $G$ and $i_0$, so that for all $v \in TM$ we have
\[
1 - C \eps \leq 
\frac
{\Vert DF(v) \Vert_{g_1}}
{\Vert v \Vert_{g_0}} 
\leq 1 + C \eps.
\]
\end{thm}

We will use that hyperbolic metrics on surfaces are determined by their \emph{Fenchel--Nielsen coordinates}.
We briefly recall their construction, see, for instance, \cite{farbmarg, buser}. 
First, we fix a \emph{pants decomposition} of the surface $M$. (See \cite[8.3.1]{farbmarg}.)
This is a maximal collection of disjoint essential simple closed curves such that no two curves are isotopic. Every pants decomposition has $3G - 3$ curves. 
When we cut the surface $M$ along these curves, we obtain a disjoint union of $2G - 2$ \emph{pairs of pants}, which are compact surfaces of genus 0 with three boundary components. The boundary curves are called \emph{cuffs}. 

When $M$ is equipped with a hyperbolic metric $g$, it is natural to take the $g$-geodesic representatives of the curves in the pants decomposition. 
A hyperbolic metric on such a pair of pants is determined by the three cuff lengths (see, for instance \cite[Proposition 10.5]{farbmarg}. 
This is proved by decomposing a pair of pants into two congruent right-angled hyperbolic hexagons by cutting along the three unique perpendiculars (called \emph{seams}) connecting each pair of cuffs.
By ``hyperbolic hexagon", we mean a hexagon in the hyperbolic space $\mathbb{H}^2$ equipped with the hyperbolic metric of constant curvature $-1$. 
The resulting hyperbolic hexagons each have three alternating sides whose lengths are half the cuff lengths of the original pair of pants. Knowing the length of every other side of such a hyperbolic hexagon determines it up to isometry (see \cite[Proposition 10.4]{farbmarg} or \cite[Theorem 3.1.7]{buser}). 

Two pairs of pants which share a common cuff length can be glued together to form a hyperbolic surface of genus 0 with 4 boundary components. 
There is freedom to twist one cuff relative to the other when gluing, encoded by \emph{twist parameters} (these will be defined more precisely in Section \ref{sec:twists}). 

Thus, given a pants decomposition, each hyperbolic metric has $3G - 3$ length parameters $\mathcal{L}_g(\gamma_1), \dots, \mathcal{L}_g(\gamma_{3G-3}) \in \R^{>0}$ and $3G - 3$ twist parameters $\alpha_1, \dots, \alpha_{3G-3} \in \R$. 
These determine the hyperbolic metric up to isometry; see, for instance, \cite[Theorem 10.6]{farbmarg}. 

These twist parameters are in turn determined by the lengths of $6G - 6$ additional closed geodesics; see, for instance, \cite[Proposition 3.3.11]{buser}. 
We will show these $9G - 9$ curves approximately determine the metric in the sense of the conclusion of Theorem \ref{thm:hyp}, using the constructions in \cite[Chapter 3]{buser}.

\subsection{Triangulating the surface using pants}\label{sec:pants}
Using
\cite[Theorem 5.1]{buser}, we 
fix once and for all a pants decomposition $\gamma^0_1, \dots ,\gamma^0_{3G-3}$ of $(M, g_0)$ where the $\gamma^0_j$ are closed geodesic with length at most $26(G-1)$.\footnote{For our purposes, it is sufficient to have a bound that depends only on the genus, and not to obtain the sharpest bound possible.}
Note also that the lengths of the $\gamma^0_j$ are bounded below by twice the injectivity radius of $g_0$.
Using this pants decomposition for $(M, g_0)$, we get a corresponding pants decomposition of $(M, g_1)$ by considering the $g_1$-geodesic representatives $\gamma^1_j$ of the curves $\gamma^0_j$. 
 By \eqref{eq:MLShyp}, we have:
\begin{equation}\label{eq:cuff-bds}
(1 - \eps) 2 i_0 \leq \mathcal{L}_{g_i}(\gamma^i_j) \leq (1 + \eps) 26(G-1)
\end{equation}
 for $i = 0, 1$ and $j = 1, \dots, 3G - 3$. 
 
Given this pants decomposition for $(M, g_i)$, we can decompose each pants into two congruent right-angled hexagons as described above.
Fix such a hexagon $H_i$ and label the consecutive sides $c_1^i, s_3^i, c_2^i, s_2^i, c_3^i, s_1^i$ such that for each $j = 1, 2, 3$, we have $l_{g_i}(c_j^i) = \frac{1}{2} \mathcal{L}_{g_i}(\gamma_{k(j)})$ for some $k(j) \in \{ 1, \dots, 3G - 3 \}$. Note that the sides labeled $c$ correspond to cuffs and those labeled $s$ correspond to seams. 

Next we decompose each hexagon into four geodesic triangles by drawing three diagonals emanating from a fixed ``anchor vertex", say, the vertex between $c^i_1$ and $s^i_1$ for concreteness. 
Label these diagonals $e^i_1, e^i_2, e^i_3$. See Figure \ref{fig}.

\begin{figure}\label{fig}
\begin{tikzpicture}[scale=5]

\coordinate (A) at (0.7,-0.1);
\coordinate (B) at (0.85,0.3); 
\coordinate (C) at (0.5,0.7);
\coordinate (D) at (-0.5,0.7);
\coordinate (E) at (-0.85,0.3);
\coordinate (F) at (-0.7,-0.1);

\draw (A) to[bend left=30] (B);
\draw (B) to[bend left=30] (C);
\draw (C) to[bend left=30] (D);
\draw (D) to[bend left=30] (E);
\draw (E) to[bend left=30] (F);
\draw (F) to[bend left=30] (A);

\fill(F) circle[radius=0.5pt];
\node[right] at ($(A)!0.5!(B)$) {$c_2$};
\node[above] at ($(B)!0.5!(C)$) {$s_2$};
\node[below] at ($(C)!0.5!(D)$) {$c_3$};
\node[left] at ($(D)!0.5!(E)$) {$s_1$};
\node[left] at ($(E)!0.5!(F)$) {$c_1$};
\node[above] at ($(F)!0.5!(A)$) {$s_3$};

\node[right] at ($(F)!0.5!(D)$) {$e_1$};
\node[above] at ($(F)!0.8!(B)$) {$e_3$};
\node[right] at ($(F)!0.6!(C)$) {$e_2$};

\draw[dashed] (F) to[bend right = 30] (D) ;
\draw[dashed] (F) to[bend left=5] (C); 
\draw[dashed] (F) to[bend left=20] (B); 

\node[below left] at (F) {$p$};

\end{tikzpicture}
\caption{Triangulating the right-angled hexagon $H$}
\end{figure}
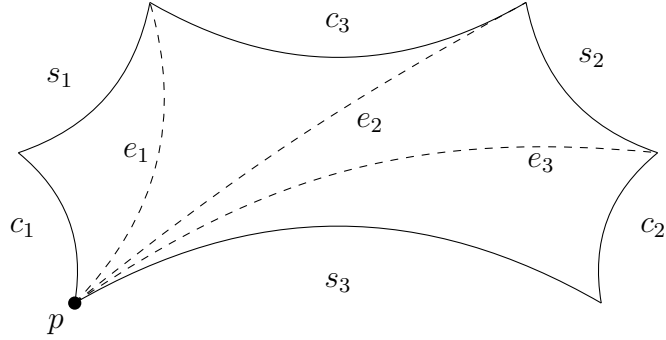

In this section, we prove that the length function hypothesis \eqref{eq:MLShyp} ensures that the resulting $g_1$-geodesic triangles have almost the same side lengths and angles as their $g_0$ counterparts. 
First we show the seam lengths of $H_0$ and $H_1$ are close: 

\begin{lem}\label{lem:seams}
For $i = 0, 1$, let $H_i \subset \mathbb{H}^2$ be right-angled hexagons with consecutive sides $c_1^i, s_3^i, c_2^i, s_2^i, c_3^i, s_1^i$. 
Fix $\eps > 0$ and suppose that 
\begin{equation}\label{eq:cuffs-approx}
1 - \eps \leq \frac{l_{g_1}(c_1^j)}{l_{g_2} (c_2^j)} \leq 1 + \eps
\end{equation}
for all $j = 1, 2, 3$. 
Then there exists $C = C(i_0, G)$ so that for all $j$, we have
\[
1 - C \eps \leq \frac{l_{g_1}(s_1^j)}{l_{g_2} (s_2^j)} \leq 1 + C \eps.
\]
\end{lem}

\begin{proof}
In light of \eqref{eq:cuff-bds}, it suffices to prove an additive version of the statement, meaning,  if $| {l_{g_0}(c_1^j)} - {l_{g_2} (c_2^j)}| < \delta$ then
$| {l_{g_0}(s_1^j)} - {l_{g_1} (s_2^j)}| < C \delta$ for some $C = C(i_0, B)$. 
If we know the length of every other side of a right-angled hyperbolic hexagon, then \cite[Theorem 2.4.1 i)]{buser} gives an explicit formula for the lengths of the other three sides:
\begin{equation}\label{eq:seams-expl}
\cosh(s^i_3) = \frac{ \cosh c^i_1 \cosh c^i_2 + \cosh c^i_3}{\sinh c^i_1 \sinh c^i_2},
\end{equation}
where we are abusing notation by writing, e.g., $c$ in place of $l_g(c)$. 
Since the $c^0_j$ are additively close to their $g_1$-counterparts, we have that $|\cosh(s_3^1) - \cosh(s_3^2)| \leq C \delta$ for some $C = C( i_0, G)$. 
By the mean value theorem applied to $x \mapsto \cosh x$ (for $x \geq 0$), we then have
 $|s_3^1 - s_3^2| \leq \frac{1}{\sinh(i_0)} C \delta$, which completes the proof.
\end{proof}

To determine the rest of the metric data of the triangulation, it remains to show the following:

\begin{lem}\label{lem:tri-data}
Consider right-angled hexagons $H_i$, where $i = 0,1$, as in Figure \ref{fig}.
Assuming \eqref{eq:cuffs-approx}, we have
\begin{enumerate}
\item The lengths of the correspoding diagonals are multiplicatively close: for all $j$, we have 
\[1 - C \eps \leq \frac{l_{g_1}(e_j^1)}{l_{g_2}(e_j^2)} \leq 1 + C \eps\] 
for some $C = C(i_0, G)$. 
\item If $\theta_i$ denotes the $g_i$-angle between two geodesic segments in Figure \ref{fig}, then 
\[ 1 - C \eps \leq \theta_1/\theta_2 \leq 1 + C \eps \] 
for some $C  = C(i_0, G)$. 
\end{enumerate}
\end{lem}

\begin{proof}
The diagonal $e_1$ is the hypotenuse of a right triangle.
By \cite[Theorem 2.2.2 i)]{buser}, we have
\[
\cosh(e_1) = \cosh(s_2) \cosh(c_3). 
\]
This means $|\cosh(e^1_1) - \cosh(e^2_1)| \leq C \eps$, and by the same mean value theorem argument as in the proof of the previous lemma, we have the lengths of the $e_1^i$ satisfy the claim. 
The sines of the angles in this right triangle are determined by the side lengths as in \cite[Theorem 2.2.2 iii)]{buser}.
Then \eqref{eq:cuff-bds} ensures these angles are from bounded above as well as away from zero.
So the angles of this triangle satisfy item (2) of the claim. 
The same arguments applies to the right triangle with hypotenuse $e_3^i$.

It thus remains to determine the sides and angles of the two ``inner triangles" of $H_i$. 
The data of the outer triangles and the edges of the hexagon determine a pair of sides together with the angle between them for each of these inner triangles.
So we can determine the remaining side from \cite[Theorem 2.2.1 i)]{buser} and remaining angles from \cite[Theorem 2.2.1 iii)]{buser}. We can conclude using very similar arguments to above. 
\end{proof}

\subsection{Stretch maps on triangles}
Let $T_0 \subset \mathbb{H}^2$ be a geodesic triangle arising from the above-described triangulation procedure of $(M, g_0)$ 
and let $T_1 \subset \mathbb{H}^2$ be the corresponding triangle coming from $(M, g_1)$. 

Given any two hyperbolic triangles $T_0$ and $T_1$ with a correspondence of edges, there is a natural ``raywise" stretch map $\sigma$ that can be defined between them. (See, for instance, \cite[p. 68]{buser}.)
We will show that when $T_0$ and $T_1$ have almost the same sides and angles as in the previous subsection, this map is $C^1$-close to the identity (in fact $C^k$ close for any $k \in \N$); in particular, we will show a quantitative version of \cite[Lemma 3.2.6]{buser}. 

To construct $\sigma: T_0 \to T_1$, we start by defining maps $S_i : (s, t) \in [0,1]^2 \to T_i$ as follows.
Let $p$ denote the ``anchor vertex" of $T$ (see above description of how we triangulated the right-angled hexagons). 
Let $c: [0,1] \to T_i$ be a constant-speed parametrization of the side of $T_i$ opposite $p$. 
For each $s \in [0,1]$, let $\eta_s: [0,1] \to T_i$ be a constant speed parametrization of the geodesic segment from $p$ to $\gamma(s)$. 
Then for any $(s, t) \in [0,1]^2$, define $S_i(s, t) = \eta_s(t)$. 
Now define the stretch map $\sigma = S_1 \circ S_0^{-1}: T_0 \to T_1$.

\begin{prop}\label{prop:sigma}
Assume \eqref{eq:cuffs-approx} holds for some $\eps > 0$. Then 
the above-defined map $\sigma: T_0 \subset \mathbb{H}^2 \to T_1 \subset \mathbb{H}^2$ restricted to the interior of $T_0$ is $C^k$-close (with respect to the hyperbolic metric on $\mathbb{H}^2$) to the identity map for any $k \in \N$.
\end{prop}

\begin{proof}
Apply an isometry of $\mathbb{H}^2$ so that $T_0$ and $T_1$ have a common ``anchor vertex" $p$. 
Let $(r, \theta)$ denote normal polar coordinates centered at $p$. 
The hyperbolic metric in these coordinates has the form $dr^2 + \sinh^2(r) d \theta^2$. 

Write $S_i(s, t) = (r(s, t), \theta(s,t))$. 
Let $e$ denote the length of the edge given by the image of $\eta_0$ and let $\alpha$ denote the angle at the vertex $\eta_0(1)$. 
Let $ r(s) = l_{g_i} (\eta_s)$. Using \cite[Theorem 2.2.1 i)]{buser}, we have
\[
\cosh  r(s) = - A \sinh(s) + B \cosh(s),
\]
where $A = \sinh e \cos \alpha$ and $B = \cosh e$. 
Let $\theta(s)$ be the angle between $\eta_s$ and $\eta_0$ at the anchor vertex $p$. Then by \cite[Theorem 2.2.1 iii)]{buser}, we have
\[
\sin \theta(s) = \sinh(s) \frac{\sin \beta}{\sinh r(s)}.
\]
Then the map $S_i$ is given by 
 $r(s,t) = t r(s)$ and $\theta(s, t) = \theta(s)$. 

By Lemmas \ref{lem:seams} and \ref{lem:tri-data}, we have that the functions $r_0(s, t)$ and $r_1(s, t)$ and the functions $\theta_0(s, t)$ and $\theta_1(s, t)$ are $C \eps$-close in the $C^k$ topology with respect to the Euclidean metric $dr^2 + d \theta^2$. 
Since $T_0, T_1$ are contained in the ball $B(p, R)$ (for some $R = R(G)$ by \eqref{eq:cuff-bds}), 
we have by the coordinate expression $dr^2 + \sinh^2(r) d \theta^2$ for the hyperbolic metric that $d_{{\rm hyp}} ( (r_0, \theta_0), (r_1, \theta_1)) \leq C d_{{\rm Eucl}}  ( (r_0, \theta_0), (r_1, \theta_1))$ for some $C = C(G)$. 

We then have $d_{C^k} ({\rm Id}, S_1 \circ S_0^{-1}) = d_{C^k}(S_1, S_2) \leq C \eps$ for some $C = C(i_0, G)$, as desired.  
\end{proof}

\begin{rem}\label{rem:sigma}
Let $g_i'$ be the hyperbolic structure whose Fenchel--Nielsen length parameters coincide with those of $g_i$, but with twist parameters all equal to zero. 
Next, note that $\sigma$ sends each edge $e^0$ of $T_0$ to the corresponding edge $e^1$ on $T_1$ by stretching by the factor $l(e^1)/l(e^0)$. 
As such, the maps $\sigma$ defined on each triangle can be glued together into a global map continuous map $\sigma: (M, g_0') \to (M, g_1')$.
Proposition \ref{prop:sigma} says that this map is almost distance preserving. 
Note, however, that while $\sigma$ is smooth on the interior of each triangle, it is not differentiable along the edges of the triangulation. We will approximate $\sigma$ by a smooth map in Section \ref{sec:smooth}. 
\end{rem}

\subsection{Twist parameters}\label{sec:twists}
Let $\gamma_1^i, \dots \gamma_{3G - 3}^i$ be the curves of the pants decomposition of $(M,g_i)$ fixed in Section \ref{sec:pants}.
When reassembling the disjoint pants into a surface, we glue two pants along a cuff which has the same length for both pants. 
The \emph{twist parameter} $\alpha_i \in \R$ of the cuff specifies the gluing.
More precisely, given a pair of pants $Y$, let $\gamma: [0, 1]/0 \sim 1 \to  Y$ be a constant-speed parametrization of a boundary component such that $\gamma(0)$ and $\gamma(1/2)$ are the locations where the seams meet the cuff. 
Let $\gamma'(t)$ denote an analogously parametrized cuff of the same length as $\gamma$ in another pair of pants $Y'$.
Then, given $\alpha \in \R$, we can glue $Y$ and $Y'$ via the identification $\gamma(t) \sim \gamma'(\alpha - t)$. This $\alpha$ is called the twist parameter. 
See \cite[(3.3.4)]{buser}. 
As mentioned above, any hyperbolic structure on a closed surface can be obtained by glueing together pants in this manner (see, for instance, \cite[Theorem 10.6]{farbmarg}). 

The twist parameters at the curves $\gamma_1, \dots, \gamma_{3G-3}$ are determined by the marked length spectrum function $\mathcal{L}_g$, as shown, for instance, in \cite[Proposition 3.3.11]{buser}.
More precisely, we have the following:

\begin{lem}\label{lem:twist1}
Fix $\eps > 0$ and suppose that  $(M,g_0)$ and $(M, g_1)$ have $\eps$-close marked length spectra as in \eqref{eq:MLShyp}.
For $j = 1, \dots, 3G - 3$ let $\alpha_j^i$ denote the twist parameter Then there is some constant $C = C(i_0, G)$ so that 
\begin{enumerate}
\item $|\alpha_j^i| \leq C$,
\item $|\alpha_j^0 - \alpha_j^1| \leq C \eps $.
\end{enumerate}
\end{lem}

\begin{proof}
We first note that the homotopy classes of curves whose lengths determine the twists are described in \cite[p. 73]{buser}.
The twist parameter about $\gamma_j$ is determined by the length of a closed geodesic $\delta_j$ transverse to $\gamma_j$ (described more precisely below) together with the length of the curve $\eta_j$ obtained by Dehn twisting $\delta_j$ about $\gamma_j$. 
Since $\delta_j$ is homotopic to a path which goes around a cuff, down two pants seams, and around another cuff (see also \cite[Figure 3.3.2]{buser}), the expression for the seams in \eqref{eq:seams-expl} implies that there is a constant $C = C(i_0, G)$ which controls the length of $\delta_j$ from above. 
Since $\eta_j$ is a Dehn twist, its length is bounded above by $l_g(\gamma_i) + l_g(\delta_i)$. 
The statement of the lemma then follows from the expressions of the twists in terms of the lengths in
\cite[Proposition 3.3.11]{buser}.
\end{proof}

Given two pairs of pants $Y$ and $Y'$ with a pair of cuffs of the same length, let $X^{\alpha}$ denote the corresponding ``X-piece" (topologically, a sphere with 4 punctures) obtained by gluing via the above-described identification $\gamma(t) \sim \gamma'(\alpha - t)$.
Then there is a map $\tau_{\alpha}: X^0 \to X^{\alpha}$ described in \cite[(3.3.7)]{buser} whose construction we now recall.

To define $\tau_{\alpha}$, it is convenient to work in Fermi coordinates on a collar neighborhood of the cuff $\gamma$.
Let $t \mapsto \gamma(t)$ be a unit-speed parametrization and 
let $(\rho, t)$ denote the point such that the perpendicular from this point to $\gamma$ has length $\rho$ and intersects $\gamma$ at the point $\gamma(t)$. 
In these coordinates, the hyperbolic metric has the form $d \rho^2 + l^2(\gamma) \cosh^2(\rho) dt^2$ (see \cite[(3.3.6)]{buser}). 
 
The map $\tau_{\alpha}$ acts nontrivially only on the region $|\rho| \leq w$, where $w$ is given as a function of $l(\gamma)$ in \cite[Proposition 3.1.8]{buser}, and is in particular bounded below by some $c = c(i_0, G)$. 
In this region, we have
\[
\tau_{\alpha}(\rho, t) = \left(\rho, t + \alpha \frac{w + \rho}{2 w}\right).
\]
(See \cite[(3.3.7)]{buser}.)
We also note the derivative of $\tau_{\alpha}$ is given by
\begin{equation}\label{eq:Dtau}
 D \tau_{\alpha} (\rho, t) = 
\begin{pmatrix}
1 & 0 \\
\frac{\alpha}{2 w} & 1.
\end{pmatrix}
\end{equation}
Hence, we obtain the following:

\begin{lem}\label{lem:twist2}
There is a constant $C = C(i_0, G)$ so that $\tau_{\alpha}$ is $C|\alpha|$-close to the identity in the $C^k$ topology. 
\end{lem}

With respect to our fixed pants decomposition $P$, let $g_i'$ be the metric which has the same length parameters as $g_i$ but twist parameters all equal to 0. 
Define maps $\tau_i : (M, g_i) \to (M, g_i')$ by glueing together above-defined $\tau_{\alpha}$ acting on disjoint collar neighborhoods of the cuffs. 
Let $\sigma: (M, g_0') \to (M, g_1')$ be the map described in Remark \ref{rem:sigma}.
Finally, define $f = \tau_1^{-1} \circ \sigma \circ \tau_0: (M, g_0) \to (M, g_1)$. 

\begin{prop}\label{prop:fmain}
There is a constant $C = C(i_0, G)$ so that the following holds. 
Consider a hexagon $H_0$ arising from the pants decomposition of $(M, g_0)$ and consider the coordinate representation of $f = \tau_1^{-1} \circ \sigma \circ \tau_0$ as a map
  $ H_0 \subset \mathbb{H}^2 \to H_1 \subset \mathbb{H}^2$. 
Then at all interior points of the triangles in $H_0$ (see Figure \ref{fig}), the coordinate representation of $f$ is $C\eps$-close to the identity map on $\mathbb{H}^2$ in the $C^k$ topology for all $k \in \N$.    
\end{prop}

\begin{proof}
We first consider the map $\tau_1^{-1} \circ \tau_0$. 
This map acts non-trivially on disjoint collar neighborhoods of the ``half-cuffs" $c_j^i$ in $H_0$.
On each collar, it is of the form $(\tau_{\alpha_j^2})^{-1} \circ \tau_{\alpha_j^1} = \tau_{\alpha_j^1 - \alpha_j^2}$. By Lemma \ref{lem:twist1} part (2) and Lemma \ref{lem:twist2}, this map is $C \eps$-close to the identity in the $C^k$ topology. 

Next, note that
\begin{align*}
d_{C^0}(\tau_0 \circ \sigma \circ \tau_1^{-1}, \tau_0 \circ \tau_1^{-1}) \leq \Vert \tau_0 \Vert_{C^1} \, d_{C^0}(\sigma \circ \tau_1^{-1}, \tau_1^{-1})
=  \Vert \tau_0 \Vert_{C^1} \, d_{C^0}(\sigma, id). 
\end{align*}
By Lemma \ref{lem:twist1} part (1) and \eqref{eq:Dtau}, we see that $\Vert \tau_0 \Vert_{C^1}$ is controlled from above by a constant depending only on $i_0$ and $G$. 
By Proposition \ref{prop:sigma} (and Remark \ref{rem:sigma}), we have $d_{C^0}(\sigma, id) \leq C \eps$ for some $C = C(i_0, G)$. 
The estimates of $C^k$-closeness for $k > 0$ follow similarly. 
\end{proof}

\subsection{Smoothing}\label{sec:smooth}
In light of Proposition \ref{prop:fmain}, in order to prove Theorem \ref{thm:hyp}, it suffices to show the following. 
\begin{prop}
There exists $C = C(i_0, G)$ such that the following holds. 
Given $\delta > 0$, there is a smooth map $F: (M, g_0) \to (M, g_1)$ which, on the interior of the triangulation of Section \ref{sec:pants}, is $C (\delta + \eps)$-close in the $C^1$ topology to the map $f = \tau_0 \circ \sigma \circ \tau_1$.  
In particular, for $\eps$ sufficiently small in terms of $C$, the map $F$ is still a diffeomorphism. 
\end{prop}

\begin{proof}
Consider the triangulation of $(M,g_1)$ fixed in Section \ref{sec:pants}. We will do a smoothing procedure in two steps: first we will smooth in neighborhoods of the vertices of the triangulation, then we will further smooth along the remaining disjoint geodesic edges.

\textbf{Step 1: smoothing at vertices.}
There exists $r = r(i_0)$ such that the balls $B(p, r) \subset (M, g_0)$ centered at vertices $p$ of the triangulation are disjoint. 
These balls are locally isometric to balls in $\mathbb{H}^2$, which we can parametrize in polar coordinates.  
Abusing notation, let $f = (f_1, f_2): B_p(r) \subset \R^2 \to \R^2 $ denote the polar coordinate representation of $f$ (where $f_i$ are $\R$-valued). 

For any $\delta < r/4$, let
 $\chi_{\delta}: \R^2 \to \R$ be a radial bump function which is $0$ on the outside of  $B_0(r/2)$, constant on $B_0(\delta)$ and with $\int_{\R^2} \chi = 1$. 
Now consider the map $f_{\delta} = ((f_1)_{\delta}, (f_2)_{\delta})$ where $(f_i)_{\delta}$ is given by the convolution $f_i * \chi_{\delta}$. 
Since
\begin{align*}
f_i * \chi_{\delta}(x) - f_i(x) = \int_{\R^2} \chi(y) (f_i(x-y) - f_i(x)) \, dy,
\end{align*}
it is apparent that
$\Vert f - f_{\delta} \Vert_{C^0} \leq L \delta$, where $L$ is a Lipschitz bound for $f$.
By Proposition \ref{prop:fmain}, we can take $L = 1 + C \eps$ for some $C = C(i_0, G)$. 
Moreover, 
\begin{align*}
\frac{d}{dx} ((f_i * \chi_{\delta}(x) - f_i(x)) = \int_{\R^2} \chi(y) (f_i'(x-y) - f_i'(x)) \, dy,
\end{align*}
since the $f_i$ are differentiable almost everywhere (away from the edges of the triangulations). 
The right hand side is bounded above by $\max_{x \in B_0(\delta)} f_i'(x) - \min_{x \in B_0(\delta)} f_i(x)$. 
In light of Proposition \ref{prop:fmain}, this is bounded above by $C \eps$ for some $C = C(i_0, G)$. 
Hence $\Vert f - f_{\delta} \Vert_{C^1(B_p(r))} \leq C (\delta + \eps)$ for some $C = C(i_0, G)$. 

Now we further modify $f_{\delta}$ so that is agrees with $f$ in a neighborhood of the boundary of $B_p(r)$.
Returning to the local chart $B_0(r) \subset \R^2$, let $\tilde \chi: \R^2 \to \R$ be a radial bump function which is $0$ outside of $B_0(r/2)$ and $1$ on $B_0(r/4)$. Note that the size of the derivative $\tilde \chi'$ is controlled by a constant depending only on $r$, and hence on the injectivity radius $i_0$. 
Now consider the function 
\begin{equation}\label{eq:avg}
\tilde f_{\delta} = \tilde \chi f_{\delta} + (1 - \tilde \chi) f. 
\end{equation}
Since $\tilde f_{\delta} - f = \tilde \chi (f_{\delta} - f)$, we see that $\Vert \tilde f_{\delta} - f \Vert_{C^1(M)} \leq C (\eps + \delta)$ for some $C$ that depends only on $i_0$ and $G$. 
Hence, replacing $f$ with $\tilde f_{\delta}$ on $B(p, r)$, we obtain a new map which is smooth on $B(p, r)$ and $C^1$-close to the original $f$. Repeating this procedure on each of the above-specified disjoint vertex neighborhoods yields a new $\tilde f$ which is $C^1$-close to $f$ and is additionally smooth on all $r$-balls centered at the vertices of our triangulation. Also, $f$ is still a diffeomorphism if this $C^1$ distance is sufficiently small. 

\textbf{Step 2: Smoothing at disjoint edge interiors.}
After doing the smoothing procedure replacing $f$ with $\tilde f$ in Step 1, we are left with a disjoint union of subsets of interiors of edges of the triangulation on which $\tilde f$ is still not smooth.
There exists a width $w$ depending only on the choice of $r$ from Step 1 and the angles between the edges of the triangulation such that the collars of width $w$ about these geodesic segments are all disjoint. 
By Lemma \ref{lem:tri-data} part (2), we conclude that $w$ depends only on $i_0$ and $G$. 

Now we work in Fermi coordinates $(\rho, t)$ of these disjoint collar neighborhoods. To smooth out $\tilde f$ on these neighborhoods, we proceed very similarly to Step 1. The only difference is that instead of working with radial bump functions, we replace $\chi$ and $\tilde \chi$ with bump functions which depend only on the $\rho$ coordinate, ie, the distance to the geodesic edge. This completes the proof.
\end{proof}

\bibliographystyle{amsalpha}
\bibliography{quantMLSdimn2}

\end{document}